\documentclass{amsart}  
 \usepackage{hyperref}   
 \hypersetup{
colorlinks = true,
} 
    
\usepackage{amsmath, amssymb, mathrsfs, amsfonts,  amsthm} 
\usepackage [all]{xy}

\newcommand {\lag} {\gamma_{\infty}}

\DeclareMathOperator{\ring}{\widehat{QH}( \mathbb{CP} ^{\infty})}

\DeclareMathOperator{\cp}{ \mathbb{CP}^{n-1}}

\newtheorem {theorem} {Theorem} [section] 
 
\newtheorem{conventions}{Conventions}
\newtheorem {remark/question} [theorem] {Remark/Question}

\newtheorem{lemma}[theorem] {Lemma} 

\newtheorem {question}  [theorem] {Question}
\newtheorem {definition} [theorem] {Definition} 
\newtheorem {proposition}  [theorem]{Proposition} 
\newtheorem {corollary}[theorem]  {Corollary}

\newtheorem {notation}[theorem] {Notation}
\newtheorem {remark} [theorem] {Remark}
\numberwithin {equation} {section} 
\DeclareMathOperator{\image}{image} 

\begin{document}
\author { Yasha Savelyev}
\title   {Bott Periodicity and stable Quantum classes}
%  \classno{53D45 (primary), 57R19, 14D21 (secondary)}    
 \begin{abstract} We use Bott periodicity to relate previously defined quantum
 classes to certain ``exotic Chern classes'' on $BU$. This provides an
 interesting computational and theoretical framework for some Gromov-Witten
 invariants  connected with cohomological field theories. This framework has  
 applications to study of higher dimensional, Hamiltonian rigidity aspects of
 Hofer geometry of $ \mathbb{CP} ^{n}$, one of which we discuss here. 
%  The form of the main
%  theorem, invites us to also speculate here on some ``quantum'' variants  of
%  Mumford's original conjecture (Madsen-Weiss theorem) on the cohomology of stable moduli space of Riemann surfaces.
 \end {abstract}
 \maketitle
 
\section {Introduction}
We study here some interactions of topology in infinite dimensions and
Gromov-Witten theory. Here the main space is
$BU$, and we show that Gromov-Witten theory completely detects its rational
cohomology. The vague philosophy is that certain parametric GW invariants are
related to Chern numbers via Bott periodicity. The first application of this is for the
study of higher dimensional aspects of Hofer geometry of
 $ \text {Ham}( \mathbb{CP} ^{n}, \omega _{st})$, and we are able to prove a certain rigidity result for the embedding
 $SU (n) \to \text {Ham}( \mathbb{CP} ^{n-1})$.  
 More geometric applications are
 given in \cite{Karea}, these concern Gromov K-area and aspects of almost Kahler
 geometry.
 
 Another 
 topological application is given in \cite{stablemaps}, where we use the main
 result of this paper to probe topology of the configuration space of stable
 maps in BU, which might be the first investigation of this kind.
% The first reason and most of the original 
% motivation comes from the study of K-area, 
% proposed by Gromov in \cite{Gromov} and further developed by Polterovich
%  \cite{polterovich}, as  an interesting way to probe global geometry of a
%  symplectic manifold via geometry of complex vector bundles over it. Our main
%  results here  give a great deal of new understanding of Gromov K-area,
%  and this is explained in \cite{Karea}. 

 More intrinsically, we get some new
 insights into Gromov-Witten invariants themselves, as through this
 ``topological coupling'' and some transcendental (as opposed to algebraic
 geometric in nature) methods we will compute some rather impossible looking
 Gromov-Witten invariants. These methods involve Bott periodicity
 theorem and differential geometry on loop groups. 
% 
% Lastly, this paper is a conceptual setup for some 
% interesting ``quantum'' variants of Mumford's original conjecture on cohomology
% of the stabilized moduli space of Riemann surfaces,  and this further fits into the
% above mentioned theme of topology via Gromov-Witten theory. 
\subsection {Outline} 
One crucial theorem in topology is the Bott periodicity
theorem for the unitary group, which is equivalent to the statement:
\begin{equation*} BU \simeq \Omega SU,
\end{equation*}
where $SU$ is the infinite special unitary group. On the space $BU$ we have
 Chern classes uniquely characterized by a set of axioms. It turns out that
 the space $\Omega SU$ also has natural \emph{intrinsic}  cohomology classes,
 characterized by axioms, but with a somewhat esoteric coefficient ring:
 $\ring$, a completed formal quantum homology ring of $ \mathbb{CP}
 ^{\infty}$, which for our choice of coefficients turns out to be the free
 Laurent polynomial algebra over $ \mathbb{Q}$ on one generator.
% , which cons finite sums \begin{equation}
% \Sum _{i} q _{i} e ^{d _{i} A}, \quad q _{i} \in \mathbb{Q}, d _{i} \in \mathbb{Z}, 
% \end{equation}
% and where $A$ is the cl
Here is an indication of how it works for $\Omega
SU (n)$, the reader may note that this is a natural extension of Seidel
representation, \cite{Seidel}. 
Consider the Hamiltonian action of the group $SU(n)$ on $\cp$. Using this, to  a cycle $f: B \to \Omega SU(n) \simeq \Omega
^{2} BSU (n)$ corresponds a family of Hamiltonian $ \mathbb{CP} ^{n-1}$ bundles
over $ \mathbb{CP} ^{1}$ trivialized over $0 \in \mathbb{CP} ^{1}$.
We may think of this family as a bundle
$$ \mathbb{CP} ^{n-1} \times \mathbb{CP} ^{1} \xrightarrow{i} P _{f} \to B $$
over $B$,  with structure group the  group  Hamiltonian bundle automorphisms of
$\mathbb{CP} ^{n-1} \times \mathbb{CP} ^{1}$, fixing the fiber over 
$ 0 \in \mathbb{CP} ^{1}$. So we have a natural embedding $$I: \mathbb{CP} ^{n-1} \times B \to P _{f}.$$

 Gromov-Witten invariants in $P _{f}$ for classes
$$A = d[line] + [\mathbb{CP} ^1] \in i_*H_2  ( \mathbb{CP} ^{n-1} \times
\mathbb{CP} ^{1}),$$ and various constraints coming from $I$, induce cohomology
classes \begin{equation} qc _{k} (P_f) \in H ^{2k} (B, QH (\cp)). \end{equation} (More technically, we are talking about parametric Gromov-Witten
invariants. The bundles $P_f$ always have a naturally defined deformation class
of fiber-wise families of symplectic forms. Although very often the total space
turns out to be Kahler in which case one can really talk about Gromov-Witten
invariants and the discussion coincides.) These cohomology classes
have analogues of Whitney Sum and naturality axioms as for example Chern classes. There is a also a partial normalization.

We  show here that these classes stabilize and induce
cohomology classes on $\Omega SU \simeq BU$. These stable cohomology classes are
closely related to classical Chern classes, and using this we prove:
 \begin{theorem} \label{theorem.chern.classes} The induced  classes
$qc _{k}$ on $\Omega SU \simeq BU$ are algebraically independent and generate
 the cohomology
%   exotic  Chern classes with
with the  coefficient ring $\ring$. Moreover, $qc _{k} = c _{k} \cdot q ^{k}$
for $q ^{k} \in \ring$. 
   \end{theorem} 
  Stabilization in this context is somewhat analogous to semi-classical
approximation in physics.  The ``fully quantum objects'' are the classes
$qc  _{k} \in \Omega SU (n)$, and they are what's important in geometric
applications, for example in Hofer geometry. We are still far from
completely computing these classes, but as a corollary of the proof of the above
we have the following.
\begin{theorem} \label{theorem.corollary}
The classes $qc  _{k}$ on $\Omega SU (n)$ are algebraically independent and
generate cohomology in the stable range $2k \leq 2n-2$, with coefficients in
$QH ( \mathbb{CP} ^{n-1})$. 
\end {theorem}
The argument actually gives a concrete method of computing above mentioned
Gromov-Witten invariants using classical Chern classes. For example,
using computation of Chern numbers of bundles over spheres in \cite[Corollary
20.9.8]{fibre.bundles} we get:
\begin{theorem} Let $f: S ^{2k} \to \Omega SU (n) \subset SU \simeq BU$, $n>k
+1$ be the unit in $\pi _{2k} (\Omega SU (n), \mathbb{Q}) \simeq \mathbb{Q}$. 
 Then there are generically $\langle c _{k}, [f] \rangle = (k-1)!$ vertical
 holomorphic curves (counted with signs) in degree $d=-1$, going through the 
 cycle $$I_*([ \mathbb{CP} ^{n-k}] \times [S ^{2k}]) \in H _{4k-2} (P _{f}).$$ 
\end{theorem}
Since $qc _{*}$ on $\Omega SU (n)$ are pulled back from classes on $\Omega
\text {Ham}( \mathbb{CP} ^{n-1})$ as a simple  topological corollary of Theorem
\ref{theorem.chern.classes} we obtain another proof of: 
\begin{theorem} [Reznikov, \cite{rez}] The natural inclusion $\Omega SU (n) \to
\Omega \text {Ham}( \mathbb{CP} ^{n-1})$ is injective on rational homology in
degree up to $2n-2$.
\end{theorem}
Reznikov's argument is very different and more elementary in nature, he also
proved a stronger topological claim, (he did not have conditions on degree) but
of course his emphasis was different, and the much more important
immediate application is to  Hofer geometry. 
\subsection {Applications to Hofer geometry}
For a closed symplectic manifold $(M, \omega)$ recall that the (positive) Hofer
length functional $L ^{+}: \mathcal {L} \text {Ham}(M ^{n}, \omega) \to \mathbb{R}$ is defined by 
\begin{equation*} \label {eq.pos.length}L ^{+} (\gamma):= \int_0 ^{1} \max (H
^{\gamma} _{t}) dt, \end{equation*} where $H ^{\gamma}: M \times S ^{1} \to
\mathbb{R}$ is a  generating Hamiltonian function for $\gamma$ normalized by the
condition \begin{equation*} \int _{M} H ^{\gamma} _{t} \omega ^{n}=0.
\end{equation*} 
\begin{conventions} The Hamiltonian vector field is defined by $\omega (X
_{H}, \cdot) = -dH (\cdot)$.
\end{conventions}
Let $$j ^{E}: \Omega ^{E} \text {Ham}(M, \omega) \to \Omega \text {Ham}(M,
\omega)$$ denote the inclusion where $ \Omega ^{E} \text {Ham}(M, \omega)$ is
the sub-level set with respect to the functional $L ^{+}$. 
 Define $$(\rho, L ^{+}): H_* (\Omega \text {Ham}( \mathbb{CP} ^{n-1})) \to
 \mathbb{R},$$ to be the function \begin{equation*} \label {eq.N} (\rho, L ^{+})
 (a) = \inf \{E| \, a \in \image j ^{E} _{*} \subset H_* (\Omega \text {Ham}(
 \mathbb{CP} ^{n-1})) \}.
\end{equation*}
Let us also denote by $i$ the
natural map $$i: \Omega SU (n) \to \Omega \text {Ham}( \mathbb{CP} ^{n-1}),$$
and by $i ^{*} L ^{+}$ the pullback of the function $L ^{+}.$
 We have an analogously defined function $$ (\rho, i ^{*} L ^{+}): H_*
(\Omega SU (n)) \to \mathbb{R}.$$  \begin{theorem} \label{thm.Hofer}
If $a \neq 0 \in H _{2k} (\Omega SU (n))$ then $ (\rho, L ^{+}) (i_*a)= 1$, 
provided that $2 \leq 2k \leq 2n-2$, where the standard symplectic form on $ \mathbb{CP} ^{n-1}$ is normalized by the
condition that the symplectic area of a complex line is $1$. Moreover we have
the following Hamiltonian rigidity phenomenon: 
\begin{equation*} (\rho, i ^{*}L ^{+}) (a) = (\rho, L ^{+})(i_* a),  
\end{equation*}
for class $a$ satisfying same conditions.
\end{theorem}
What is already interesting is that $(\rho, L ^{+}) (i_*a) \neq 0$, as $ \text
{Ham}( \mathbb{CP} ^{n-1})$ is  a  very complicated infinite dimensional metric space,
and sublevel sets $\Omega ^{E} \text {Ham}( \mathbb{CP} ^{n-1}) $ may
have interesting homology for arbitrarily small $E$. However this does not
happen  for $M=S ^{2}, \mathbb{CP} ^{2}$, as $ \text {Ham}(S^2)  \simeq
SO (3)$ by a theorem of Smale and $  \text {Ham}( \mathbb{CP} ^{2}) \simeq PSU (2)$ by a
theorem of Gromov \cite{Gromov2}, and so if $k>0$,  the above theorem give a
lower bound for $E$. On the other hand in the case $k=0$, we have a lower bound
for $ E$ because Seidel representation \cite{Seidel} for $M= S ^{2}, \mathbb{CP}
^{2}$ is injective.

We can actually replace $ SU (n)$ by $ PU (n)$ and allow $k=0$ but this will not
give anything new, as the case $k=0$ is already covered by  McDuff-Slimowitz
\cite{MSlim} in more generality, see also \cite{Karea}, \cite{Entov} for 
very interesting related work in the case $k=0$.
\begin{question} Does the rigidity statement:
\begin{equation*}  (\rho, i ^{*}L ^{+}) (a) = (\rho, L ^{+})(i_* a),
\end{equation*}
hold for the full Hofer length functional $L$, obtained by integrating the full
oscillation $\max H ^{\gamma} _{t} - \min H ^{\gamma} _{t}$? Our argument does break down in this
case, as it is not clear how to simultaneously bound both $L ^{+}$ and $L ^{-}$.
\end{question}
\subsection*{Potential applications to topology of smooth manifolds.} Although
our focus here is essentially on a single symplectic manifold $( \mathbb{CP}
^{n}, \omega _{st})$. The quantum classes we study here and Theorems
\ref{theorem.chern.classes}, \ref{theorem.corollary} could have some other
unexpected applications. 

 For a smooth manifold $Z ^{n}$ we may define quantum variants of
Pontryagin classes as follows. The complexified tangent bundle $TZ \otimes \mathbb{C}$
induces a classifying map $Z \to BU (n)$ and consequently a map $\Omega ^{2} Z
\to \Omega ^{2} BU (n)$. For simplicity we restrict to identity components of the
iterated loop space, in which case the map is really $$ \Omega ^{2} Z \to \Omega
^{2} BSU (n) \simeq \Omega SU (n).$$
 Our quantum classes naturally live in cohomology of the
right hand side. We may then ask how their pull-backs to cohomology of $\Omega
^{2}Z$ depend on the smooth structure, or more directly on the tangent bundle.
Actually since rational Pontryagin classes are known to be topologically 
invariant by a Theorem of Novikov, it is plausible although we have not verified this, that at
least in the stable range (as in Theorem \ref{theorem.corollary})
 quantum classes should also be topologically invariant, (our coefficients at
 least at this time are $ \mathbb{Q}$-vector spaces). The case of unstable
 range seems more mysterious. Even more mysterious is what happens when we pass
 to some ``partial stable map compactifications'' of $\Omega ^{2} Z$ as we
 attempt to do in \cite{stablemaps}.  
 \subsection* {Acknowledgements} I would like to thank Leonid
 Polterovich and Tel Aviv university for
 inviting me, and providing with a friendly atmosphere
 in which to undertake some thoughts which led  to this article. In
 particular I am grateful to Leonid for compelling me to think about Gromov
 K-area. I also thank Dusa McDuff for helping
 me clarify some confusion in an earlier draft and Alexander Givental,
 Leonid and Dusa for comments on organization and content as well as the
 anonymous referee for some brilliant comments and suggestions. 
\section {Setup} \label{section.setup}
This section discusses all relevant constructions, and further outlines the
main arguments.  
\subsection {Topological preliminaries} In a few instances, we will need to use
this characterization of homology of $H$-spaces.  
 \begin{theorem} \label
{thm.milnor.moore} [Milnor-Moore \cite{MM}, Cartan-Serre \cite{CS}] \label{Milnor-Moore}  
Let $X$ be a connected homotopy associative
 $H$-space. 
% Denote by $ \pi_* (X, \mathbb{Q}) \subset H_* (X, \mathbb{Q})$ the  Lie
% sub-algebra  of the associated algebra of the ring, generated by the image of
% the Hurewitz map $h: \pi_* (X) \to H_* (X, \mathbb{Q})$ 
Denote by
 $ \mathcal {U} (\pi_* (X, \mathbb{Q}))$ the universal enveloping algebra of 
 the Lie algebra  $\pi_* (X, \mathbb{Q})$ with respect to the Samelson product.
 Then \begin{equation*} H_* (X, \mathbb{Q}) \simeq \mathcal {U} (\pi_* (X, \mathbb{Q})),
\end{equation*}
as rings (in fact as Hopf algebras).
\end{theorem}
If we apply this theorem to identity component of the based loop space $\Omega
X$, for $X$ a connected $H$-space, we
get that $H_* (\Omega X, \mathbb{Q})$ is a graded commutative tensor algebra of
$\pi_* (\Omega X, \mathbb{Q})$. 
This is because the Lie algebra $ \pi_* (\Omega X, \mathbb{Q})$, whose product
is the Samelson product, is Abelian.  
\subsubsection* {Splitting principle for $BU$}
The classical formulation of this is that for a complex vector bundle $E \to B$
one can find a space and map $s: F \to B$, with $s ^{*}$ injective on cohomology s.t. $s
^{*} E$ splits into sum of line bundles. 
\subsubsection* {Bott periodicty}
We will use only the following version of Bott periodicity for the unitary
group: 
\begin{equation*} BU \simeq \Omega SU,
\end{equation*}
probably the most well known exposition of this is in Milnor's
\cite{Morse.theory}.
\subsubsection* {Reduced $K$-theory groups of a space $X$}
These are the groups $ \widetilde{K} (X)$ of homotopy classes of maps $f: X \to
BU \simeq \Omega SU$. The group structure is induced by the natural (in homotopy
category) map $\Omega SU \times \Omega SU \to \Omega SU$. This map can either be taken to be
concactenation product of loops or point-wise product of loops with respect to
the group structure on $SU$. 

A vector bundle $E ^{n} \to X$, has a classifying map $f_E: X \to BU (n)
\hookrightarrow BU$, and we will say that $f_E \to BU$ classifies the reduced $K$-theory class of
$E$, which is a somewhat more concise way of saying that it classifies the
stable equivalence class of the vector bundle $E \to X$, with bundles $E_1, E_2$
called stably equivalent if $E_1 \oplus \epsilon^n \simeq E_2 \oplus \epsilon
^{m}$ for some $n,m$, with $ \epsilon ^{n}$ denoting trivial rank $ n$
complex vector bundle.

\subsection* {Quantum homology} 
 In definition of quantum homology $QH (M)$, or various Floer
homologies one often uses some kind of Novikov ring coefficients with which $QH
(M)$ has a special grading, making quantum multiplication graded.  For us $QH
(M)$ serves as a coefficient ring for some cohomology classes and in the instance of this paper  $M$ is always $
\mathbb{CP} ^{n}$, so we will take the simplest working version of the quantum homology ring, which will
be ungraded. As a note to experts on this subject, this basically corresponds to
setting all the formal parameters to be 1, and this always works in the monotone
case.
(There is indeed a constraint in the word ``works'' since Whitney
sum formula to be described below relies on properties of quantum product.)

This choice also becomes more natural when we come to $QH ( \mathbb{CP}
^{\infty})$, as that can no longer have any natural (from point of view of
quantum product) grading anyway. 
\begin{definition} For a monotone symplectic
manifold $ (M, \omega)$, $\omega=k c_1 (TM)$, $k>0$, we set $QH (M) = H_* (M, \mathbb{Q})$, which we think of as ungraded vector space, and hence drop the subscript $*$.
\end{definition}
\subsection {Quantum product on $QH (M)$}
\begin{conventions} An $\omega$-compatible almost complex
structure $J$ satisfies: $$\omega (\cdot, J \cdot) >0.$$
\end{conventions}
 Let $
(M, \omega)$ be monotone, for integral generators $a,b \in H_* (M)$, this is the product defined by
\begin{equation} a*b = \sum _{A \in H_2 (M)} b_A  \in QH (M),
\end{equation}
where $b _{A} \in H_* (M)$ is the homology class of the evaluation pseudocycle
from the  moduli space of marked,
$J$-holomorphic, class $A$ curves intersecting  pseudocycles representing
$a,b$,  for a generic $\omega$ tamed $J$. This sum is finite in the monotone
case, by Gromov compactness, since dimension
and hence  bounds on Chern numbers of holomorphic spheres, give area bounds. The
product is then extended to $QH (M)$ by linearity. For more technical details see \cite{MS}. 
\subsection {The non-unital ring $QH( \mathbb{CP} ^{\infty})$ and its
completion} The following limiting procedure works for any choice of coefficients for quantum homology so
long as we forget the grading, and is particularly simple with our choice.
But interestingly, it does not dualize to quantum cohomology, so in this
example the quantum homology picture is forced on us.

We may identify $QH (
\mathbb{CP} ^{n-1})$ as a vector space with the space of degree $n$ polynomials
over $ \mathbb{Q}$, with no constant term: $$QH ( \mathbb{CP} ^{n-1})=  P
(n)=\{\sum _{1 \leq i \leq n} c_i  q ^{i}\}, \, c^{i} \in \mathbb{Q}.$$ If we
set $q ^{i}=[ \mathbb{CP} ^{i-1}]$,  then
\begin{equation} q  * q ^{j} = q ^{j+1}, 
\end{equation}
with $ *$ the quantum product,
so long as $1+j \leq n$, and by associativity $q ^{i} \ast q ^{j} = q ^{i+j}$,
if $i+j \leq n$.
%  with respect to the
% quantum product: $q ^{m}  = q ^{m \bmod n+1}$ (if we take convention $q ^{0} =
% [ \mathbb{CP} ^{n}]$) and $ q^m= [ \mathbb{CP} ^{m-1}]$, if $0 \leq m-1 \leq n$.
% In other words the formal product of $p_1,p_2$ in $P (n)$ is identified with
% actual quantum product if $\deg p_1 + \deg p_2 \leq n+1$, with $ \deg$ denoting
% polynomial degree.
This description of quantum product in $ \mathbb{CP} ^{n-1}$ of course is well
known, since there is a unique degree 1 holomorphic curve in $  \mathbb{CP}
^{n-1}$ through any two points in general position, and higher degree curves
cannot contribute to the quantum product for dimensional reasons. So we get:
\begin{lemma}  
% The direct limit of vector spaces 
% \begin{equation*} \lim _{n} QH ( \mathbb{CP} ^{n-1}),
% \end{equation*}
%  the non-unital ring of polynomials in $q$ over $ \mathbb{Q}$, with no
% constant term, and 
The linear map
\begin{equation*} i: QH ( \mathbb{CP} ^{n-1}) \to \lim _{n} QH ( \mathbb{CP}
^{n-1}) \simeq \mathbb{Q} [q]/ \mathbb{Q},
\end{equation*}
satisfies $i(a *b)=i(a) \cdot i (b)$, where $\cdot$ is the polynomial product in
$ \mathbb{Q} [q]/ \mathbb{Q}$, provided the
polynomial degree of $a$ and $ b$ adds up to at most $n$, 
\end{lemma}

In particular, we may take $ \mathbb{Q} [q]/ \mathbb{Q}$ as the definition of
quantum homology ring $QH ( \mathbb{CP} ^{\infty})$. We may
naturally complete this to a group ring of $ \mathbb{Z}$ over $ \mathbb{Q}$,
(i.e. the algebra of Laurent polynomials over $ \mathbb{Q}$) and we define this to be $\ring$. The identity
denoted by $ \textbf{1}$ informally can be thought to represent $[ \mathbb{CP} ^{\infty}]$, and the inverse to $q$ can be informally thought to be
represented by the infinite dimensional codimension two ``cycle Poincare dual''
to $[\mathbb{CP} ^{1}] \subset \mathbb{CP} ^{\infty}$. The reader may observe
as an amusing point that with these heuristics $\ring$ is the natural quantum
homology ring of $ \mathbb{CP} ^{\infty}$ if one allows finite codimension
cycles, (forgetting analysis for the moment).
\subsection {Quantum characteristic classes of $ \Omega  \text {Ham}(
 \mathbb{CP} ^{n}, \omega _{st})$} \label{sec.quant.classes} 
%  By Thom's theorem rational homology of a manifold $X$ is generated by cycles 
%  
%   $ \Omega \text {Ham}(M, \omega)$  is
%  generated by cycles: $f: B ^{k} \to \Omega \text {Ham}(M, \omega)$, where $B$ is a closed
% oriented smooth manifold, which we may also assume to be smooth
 By Theorem \ref{thm.milnor.moore}, $H_*(\Omega \text {Ham}(M, \omega))$
is \emph{freely} generated via
Pontryagin product by rational homotopy groups, which we may represent by smooth
maps $f: S ^{k} \to \Omega \text {Ham}(M, \omega)$,  (the associated
 map $f: S ^{k} \times S^1 \to \text {Ham}(M, \omega)$ is smooth.)  
 In particular, the relations between cycles in homology can  be represented
by smooth (in fact cylindrical) cobordisms. The cycles as above
will be called \emph{smooth}.
 
 In our case $M$ is
 always $(\mathbb{CP} ^{n}, \omega _{st})$, with $ \omega _{st}$ the
 Fubini-Study symplectic form normalized by $ \omega _{st} ([line])=1$. 
 Here the standard integrable $j$ is regular for all $A$. Let us then specialize
 to this case, referring the interested reader to
 \cite{GS}, for the general case. In fact we will specialize to the identity component of $\Omega \text {Ham}( \mathbb{CP} ^{n})$.
%  We now give a brief overview of the construction of classes $$qc  _{k} \in H ^{k} (\Omega \text
%  {Ham}(M, \omega), QH (M)) ,$$ for a monotone symplectic manifold $(M,\omega)$,
%  originally defined in \cite{GS}. The reader may note that this construction
%  extends Seidel representation, \cite{Seidel}.  
%  Since our coefficients have no
%  torsion these classes are specified by a map of Abelian groups \begin{align*} \Psi: H_* (\ls, \mathbb{Q}) \to QH (M),
% \end{align*}
% by setting  
% \begin{equation*} qc _k ([f]) = \Psi ([f]),  
% \end{equation*}
% for $f: B \to \Omega \text {Ham}(M, \omega)$ a cycle. However, purely with this
% point of view we lose the extra structure of the Whitney sum formula.

Given this, to a smooth cycle $f: B ^{2k} \to \Omega \text {Ham}( \mathbb{CP}
^{n})$, we associate a smooth bundle $$p: P _{f} \to B,$$ with 
\begin{equation} \label {eq.definitionPh} P _{f}=B \times \mathbb{CP} ^{n}
\times D ^{2} _{0} \bigcup B \times \mathbb{CP} ^{n} \times   D ^{2} _{\infty}/
\sim,
\end{equation}
where $ (b, x, 1, \theta)_0 \sim (b, f _{b, \theta}(x), 1, \theta)
_{\infty}$,
using the modified polar coordinates $(r, 2 \pi
 \theta)$. The orientation on $D ^{2} _{\infty}$ is taken to be standard
 orientation as a domain in $ \mathbb{C}$ while the orientation on $D ^{2}_0$ is
 the opposite orientation. This is the same convention as \cite{virtmorse}
 (unfortunately not stated there) but opposite to \cite{GS}. The fiber $ X _{b}$
 over $ b$ is a Hamiltonian bundle $ \mathbb{CP} ^{n} \hookrightarrow X _{b}
\xrightarrow{\pi} \mathbb{CP} ^{1}$, with structure group the group of bundle
automorphisms of $ \mathbb{CP} ^{n} \times \mathbb{CP} ^{1}$ (over identity on
the base), fixing the fiber over $0 \in \mathbb{CP} ^{1}$.  
% 
%  Fix a family $ \{j _{b,z}\}$, $b \in B$, $z \in \mathbb{CP}^1$ of almost complex
%  structures on $M \hookrightarrow P _{f} \to B \times \mathbb{CP}^1$ fiber-wise compatible with $\omega$. 
\begin{definition} \label{def.compatible} A family $ \{ J _{b}\}$ of almost
complex structures on fibers $X _{b}$, ($J _{b}$ is an almost complex
structure on $ X _{b}$) is called \emph{admissible} if:
\begin{itemize} 
  \item The natural map $\pi: (X_b, J_b) \to \mathbb{CP} ^{1}$ is
  $J_b$-holomorphic for each $b$.
  \item   $J _b$ preserves the vertical tangent bundle of $ \mathbb{CP} ^{n}
\hookrightarrow X_b \to \mathbb{CP} ^{1}$ and restricts to the standard
integrable $j$ on $ \mathbb{CP} ^{n}$. 
\end{itemize}
\end{definition}
The importance of this condition is that it forces bubbling to happen in the
fibers, where it is controlled by monotonicity of $( \mathbb{CP} ^{n}, \omega)$,
i.e. standard arguments in Gromov-Witten theory extend to this parametric case.
 
Since we are working with the identity component of $\Omega \text {Ham}(
\mathbb{CP} ^{n})$, the fibers $X_b$ are Hamiltonian bundle diffeomorphic to
$X= \mathbb{CP} ^{n} \times \mathbb{CP}^1$, although not naturally. The group
of $ \text {Ham}( \mathbb{CP} ^{n})$-bundle automorphisms of $X$ 
acts trivially on  homology, this follows by  \cite[Theorem
1.16]{symp.str.fiber}. 
In particular a section class $A$ in $H_2 (X _{b})$ is uniquely characterized by ``degree'' $d$, \begin{equation} A= d[line] +[\mathbb{CP}^1].
\end{equation} 
% In other words 
% by standard analysis, 

The classes we now define ``measure'' quantum self
intersection of a natural submanifold $I(B \times  \mathbb{CP} ^{n}) \subset P
_{f}$, which in terms of \eqref{eq.definitionPh} just corresponds to
inclusion into the first half of the union. The entire quantum self intersection
is captured by the total quantum class of $P _{f}$. Let $$\mathcal {M} (P _{f}, d, \{J _{b}\})$$ denote the moduli space 
  of tuples $ (u, b)$, $u$ is a $J _{b}$-holomorphic section of $X
_{b}$ in degree $d$. The virtual dimension of this space is given by the
Fredholm index: \begin{equation} \label {eq.dim} 2n + 2k + 2 \langle c_1
^{vert}, {A} \rangle = 2n + 2k + 2d \cdot (n+1) .
\end{equation}

We define $qc _{k} \in H ^{2k} (\Omega \text {Ham}( \mathbb{CP} ^{n}, \omega),
QH ( \mathbb{CP} ^{n}))$ as follows: \begin{equation} \label {eq.def.ch.class}
\langle qc_k, [f] \rangle = \sum_ {d \in \mathbb{Z}} b _{d}. \end{equation} 
%     \item $H ^{sect}_2 (X)$ denotes the section homology classes of $X$.
%     \item $ \mathcal {C}$ is the coupling class of Hamiltonian fibration
%  \begin{equation} \label {eq.hamfibration} M
%     \hookrightarrow P _{f} \to B \times \mathbb{CP}^1, \text {  see
%     \cite [Section 3]{kedra-2005-9}. }
% \end{equation}   
% Restriction  of $ \mathcal {C}$ to the fibers $X \subset P
%     _{f}$ is uniquely determined by the condition \begin{equation} \label {eq.calC} i^* ( \mathcal {C})= [\omega],
% \hspace {15pt} \int_M { \mathcal {C}}^ {n+1}=0 \in H^2 (\mathbb{CP}^1). 
% \end{equation} 
% where $i: M \to X$ is the inclusion of fiber map, and the integral above denotes
% the integration along the fiber map for the fibration $\pi: X \to \mathbb{CP}^1$.     
%      \item  The map $j_*: H ^{sect}_2(X) \to H_2(P_f)$ is
%  induced by inclusion of fiber. 
Where $b _{d} \in H_* ( \mathbb{CP} ^{n})$ is defined by
  duality: \begin{equation*} b _{d} \cdot _{ \mathbb{CP} ^{n}} c = ev _{d}
  \cdot _{B \times \mathbb{CP} ^{n}} [B] \otimes c, 
 \end{equation*} and where  
\begin{align*} ev _{d}: \mathcal {M}  (P _{f}, d, \{J _{b}\}) \to B
\times \mathbb{CP} ^{n} \\ ev _{d} (b, u) = (b, u (0)),
\end{align*}
 and $\cdot _{M}, \cdot _{B \times M}$ denote the 
intersection pairings in $ M$, respectively $ B \times M$. 
The sum  \eqref{eq.def.ch.class}, is  finite and only $d<0$ contribute for
dimensional reasons.
% \item The family $ \{J _{b}\}$ is $\pi$-compatible in the sense
% above.
% 
%  \item  $c_ { \textrm {vert}} \in H ^{2} (P _{h})$ is the first Chern class of
%  the bundle of vectors in $TP _{h}$ tangent to $M$, in the sense of the other
%  natural fibration $M \hookrightarrow P _{h} \to B \times \mathbb{CP}^1$. 
%  \item $\mathcal {C}$ is the coupling class of the Hamiltonian
%  fibration $M \hookrightarrow P_{h} \xrightarrow{\pi} B \times \mathbb{CP}^1$, 
%  (see [Definition 3.1] \cite{GS}). The  restriction of $ \mathcal {C}$ to the fibers $X
%  \subset P _{h}$ is uniquely determined by the condition 
%  \begin{equation} \label {eq.calC} i^* ( \mathcal {C})= [\omega],
% \hspace {15pt} \int_M { \mathcal {C}}^ {n+1}=0 \in H^2 (\mathbb{CP}^1). 
% \end{equation}
% where $i: M \to X$ is the inclusion of fiber map, and the integral above denotes
% the integration along the fiber map for the fibration $\pi: X \to \mathbb{CP}^1$. 
% \end{itemize}
% \begin{remark} The sum \eqref{eq.def.ch.class} is finite: dimension restriction gives a
% bound on $\langel c_1 (T ^{vert} P_f), \widetilde{A} \rangle$, where $T ^{vert}
% P _{f}$ is the vertical tangent bundle of the fibration \eqref{eq.hamfibration}.
% Since we are in the monotone setting this gives an ``area'' bound on $
% \widetilde{A}$, and there are only finitely many homology classes $
% \widetilde{A}$, with the given area bound. See \cite{GS}
% \end{remark}
% 
%    
%    , as
% explained in \cite{GS}. 
The fact that $qc _{k}$ is well defined with respect to various choices: the
representative $[f]$, and the family $ \{J _{b}\}$, is described in more detail
in \cite{GS}, however this is a very standard argument in Gromov-Witten theory:
deformation of this data gives rise to a cobordism of the above moduli spaces. 
\begin{remark} While working with smooth cycles is formally
adequate, it is sometimes necessary to work with general representatives
for singular homology classes, coming from smooth chains, this does not present
significantly more analytical difficulty and most of this analysis appears in
\cite{Michael}.
\end{remark}

%  If $T
% ^{vert} P _{f}$ denotes the vertical tangent bundle of \eqref{eq.hamfibration},  then the
% natural restrictions on the dimension of the moduli space 
% \begin{equation} 2n+2k +2\langle c_1 (T ^{vert} P _{f}), \widetilde{A} \rangle,
% \end{equation}
% give rise to bounds on $\langle c_1 (T ^{vert} P _{f}), \widetilde{A} \rangle$,
% i.e. to degree  $d$, where $ \widetilde{A} = [\mathbb{CP}^1] + d [line]$, as a class in
% $X= M \times \mathbb{CP}^1$. Consequently only finitely many such classes can
% contribute. 
% 
% Bundles of the type $P _{f}$ above have a Whitney sum operation.
% Given $P _{f_1}$, $P _{f_2}$ we get the bundle $P _{f_1} \oplus P _{f_2} \equiv P _{f_2 \cdot f_1}$,  where $f_2 \cdot f_1$ is the pointwise multiplication of the maps  $f_1, f_2: B \to \ls$, using the
% natural topological group structure of  $\ls$. Geometrically this corresponds
% to doing connected sum on the fibers \cite{GS}[Section 4.4]  (which themselves
% are fibrations over $\mathbb{CP}^1$), and the set of (suitable) isomorphism  classes of such bundles over
% $B$ form an Abelian group $\mathcal {P} _{B}$,  the group of homotopy
% classes of $$f: B \to \ls.$$
% %   \begin{theorem} [Bott] There is a natural inclusion $b: BU (n) \to \Omega SU
% % (2n)$ which induces an isomorphism on homotopy groups in dimensions  $\leq 2n$.
% % \end{theorem} 
% For
% $P_f \in  \mathcal {P} _{B}$ we set 
% \begin{equation} qc  _{k} (P _{f}) = f ^{*} qc  _{k}.
% \end{equation}
 We  now state the properties satisfied by these
 classes, these are
 verified in \cite{GS}.  
%  We may state   for these classes in terms of
%  their pull-backs under maps $f: B \to \Omega \text {Ham}( \mathbb{CP} ^{n})$ as
%  is  customary for characteristic classes, which is exactly the
%  approach in \cite{GS}. 
 Since our classes are already defined  on the universal
 level, we may proceed as follows:
 Let $\overline {qc}= \sum _{i=0} ^{\infty} qc _{i}$, be the formal sum (the
 total quantum class of $\Omega \text {Ham}( \mathbb{CP} ^{n})$), and $$W:
 \Omega \text {Ham}( \mathbb{CP} ^{n}) \times \Omega \text {Ham}( \mathbb{CP} ^{n}) \to \Omega \text {Ham}( \mathbb{CP} ^{n})$$ be 
the natural concatenation product, or equivalently (in the homotopy category)
the product induced by pointwise multiplication of loops, using topological
group structure of $ \text {Ham}(M, \omega)$. Then we have:
% \emph{Quantum classes} are a sequence of functions $$qc  _{k}: \mathcal {P}
% _{B} \to H ^{k} (B, QH (M))$$ satisfying the following axioms: 
% \begin{axiom} [Naturality] \label{functoriality/}  For a map $g: B_1 \to B_2$:
% $$ g ^{*} qc _k (P _{2}) = qc _k g ^{*} (P _{2}). $$ 
% \end{axiom}
\begin{itemize} 
  \item 
\begin{equation} W ^{*} \overline {qc} =
\overline {qc} \otimes \overline {qc}.
\end{equation}
% where $\cup$ is the usual cup product of cohomology classes with coefficients in
% the ring $QH ( \mathbb{CP} ^{n})$.
% If $P,
% P_1, P _{2} \in \mathcal {P} _{B}$ and $P= P_1 \oplus P_2$, then
% \begin {equation*} qc  ( P) = qc  ( P _{1}) \cup qc  ( P _{2}),
% \end {equation*} where $ \cup$ is the cup product of cohomology classes with
% coefficients in the quantum homology ring $QH(M)$ and $qc ( P)$ is
% the total characteristic class
% \begin{equation} 
% qc (P) = qc _0 (P)+\ldots+ qc _m ( P),
% \end{equation}
% where $m$ is the dimension of $B$. 
\item
$ \langle qc _{0}, [pt] \rangle=
\textbf{1} = [ \mathbb{CP} ^{n}] \in QH ( \mathbb{CP} ^{n})$.
\end{itemize}
\begin{notation} From now on we will be using shorthand $c[f]$ for the
evaluation of cohomology class $c$ on a homology class $[f]$.
\end{notation}
\begin{corollary} \label{corollary.homo} For the
Pontryagin product  \begin{equation} f_1 \star f_2: B_1 \times B_2 \to \Omega
\text {Ham}( \mathbb{CP} ^{n}), 
\end{equation}
of a pair of cycles $f_1, f_2: B ^{2k}_1, B ^{2l}_2 \to \Omega \text {Ham}(
\mathbb{CP} ^{n})$, 
\begin{equation} qc _{k+l} [f_1 \star f_2]  = qc _{k}
[f_1]  \ast qc_l[f_2] .
\end{equation} 
\end{corollary}
We now restrict our attention to the subgroup $i_n: \Omega SU (n) \subset \Omega
\text {Ham}( \mathbb{CP} ^{n})$. We set $$qc_k  ^{n}= i _{n} ^{*} qc _{k}.$$

% The map $\Psi: \Omega SU (n) \to QH (
% \mathbb{CP} ^{n-1})$ will be denoted by $\Psi ^{n}$. Note that in this case
%  \begin{equation} \label {eq.normalization} \Psi ^{n} ( [pt])= \textbf{1} = [
%  \mathbb{CP} ^{n-1}],
% \end{equation}
% since $SU (n)$ is simply connected.
We will proceed to induce these characteristic classes on the limit $\Omega SU$.
This will allow us to arrive at a true normalization axiom, completing the axioms in that setting,
and will also allow us to make contact with the splitting
principle on $BU$.  Let 
\begin{align*} i: \Omega SU (n) \to \Omega SU (m), \\
i (\gamma) (\theta) = \begin{pmatrix}  \gamma (\theta) & \\ & 1 \\
& & \ldots \\ 
& & & 1\\
 \end{pmatrix}
\end{align*}
for $m>n$, and
\begin{align*} j:  \mathbb{CP} ^{n-1} \to  \mathbb{CP} ^{m-1},\\
j ( [z_0, \ldots, z_n]) = [z_0, \ldots, z_n, 0, \ldots, 0]
\end{align*}
 be compatible inclusions.
Here is the main step, proved in Section \ref{section.proofs}.
 \begin{proposition}  \label{thm.induction} For a cycle $f: B ^{2k} \to \Omega
 SU (n)$  \begin{equation}  \label
{eq.induction} qc _{k} ^{m} [i \circ f]= j_* (qc _{k} ^{n} [f]) \in QH (
\mathbb{CP} ^{m-1}),
\end{equation}
 for $2k$ in stable range $ [2, 2n-2]$.  
\end{proposition}
%  The above proposition is indeed fairly surprising. For even if one believes
%  that the bundles induced by $j \circ f$, $f$ have
%  basically the ``same twisting'', the ambient
%  spaces (fibrations, fibers) are
%    certainly different and this same twisting can manifest itself in different
%    looking elements of quantum homology.  And in fact this different
%    manifestation is certain to happen unless $\Psi ^{n} ([f])$ is of the form
%    $c \cdot e ^{i} \in QH_* ( \mathbb{CP} ^{n-1})$, i.e. unless the only
%    contribution to $\Psi ^{n} ([f])$ is coming from the section class $
%    \widetilde{A} \in H_2 ( \mathbb{CP} ^{n-1} \times \mathbb{CP}^1)$ with $ \mathcal C
%    (\widetilde{A})=\omega ( \widetilde{A})= -1$. For if its not, the above
%    equality is impossible by dimension considerations, (this is a
%    straightforward exercise involving definition of $\Psi$.) 
   \begin{corollary} 
   There are induced cohomology classes $qc ^{\infty}_{k} \in H ^{*}(\Omega SU,
   \ring)$ by setting $qc ^{\infty}_0[pt] = \textbf{1}$, and $qc ^{\infty} _{k}
   [f] = qc ^{n} _{k} [f]$ for any $n> \deg[f]/2+1$. 
%  \begin{align*} qc _{k} ^{\infty} \in H ^{2k}
%  (BU, \widehat{QH} ( \mathbb{CP} ^{\infty})) \\
%  qc_{k} ^{\infty} ([f]) = \Psi ([f]), 
% \end{align*}
%  for $f: B ^{2k} \to BU$ a cycle.
%  These classes satisfy $$ j_* \left
%  (c ^{q} _{2k} ([f]) \right) = c ^{q} _{2k} (i \circ f)  \text { for }0 < 2k
%  \leq 2n- 2 \text { for a cycle } f: B ^{2k} \to \Omega SU (n),$$ where $$i: \Omega SU (n) \to \Omega
%  SU,$$ is the natural inclusion and where $$j_*: QH_* ( \mathbb{CP} ^{n-1})
%  \to \ring$$ is the universal map.
\end{corollary}
% Let  $ \widetilde{K} (B)$ denote the reduced K-theory group
% of $B$. 
% The group of homotopy classes of maps $f: B \to \Omega SU$, is just $
% \widetilde{K} (B)$: the reduced K-theory of $B$.
Let  $\overline {qc} ^{\infty}= \sum _{i=0} ^{\infty} qc ^{\infty}_{i}$, be
the formal sum (the total quantum class of $\Omega SU$), and $$W:
 \Omega SU \times \Omega SU\to \Omega SU$$ be the natural concatenation
 product.
\begin{theorem} \label{thm.axioms}
The classes $qc _{k} ^{\infty} \in H ^{2k} (\Omega SU, \ring)$ satisfy
and are determined by the following properties:
\begin{itemize} 
 
  \item \emph { \textbf{Whitney sum formula}}:\begin{equation} 
  \label{thm.whitney.sum} W ^{*} \overline {qc} ^{\infty} = \overline {qc}
  ^{\infty} \otimes \overline {qc} ^{\infty},
\end{equation}
where $\cup$ is the usual cup product of cohomology classes with coefficients in
the ring $\ring$.
\item \emph { \textbf{Normalization}}:
% If $P, P_1, P_2
% \in \widetilde{K} (B)$ and $P= P_1 \oplus P_2$, then
% \begin {equation*} qc ^{\infty} ({ P})  = qc ^{\infty} (P _{1}) \cup qc
% ^{\infty} ( P _{2}), \end {equation*} where $ \cup$ is the cup product of cohomology classes with
% coefficients in the ring $\ring$ and $qc ^{\infty}
% (P)$ is the total characteristic class
% \begin{equation} \label {eq.total.class}
% qc ^{\infty}(P) = \textbf{1} +\ldots+ qc ^{\infty} _{m}( P) ,
% \end{equation}
% where $m$ is the dimension of $B$.
% [Normalization] 
 If $f _{l}:
\mathbb{CP} ^{k}  \to BU \simeq \Omega SU $
 is the classifying map for the reduced $K$-theory class of the canonical
 bundle then:
 % and $i _{B}: BU \to \Omega SU$ is the Bott map, then
 \begin{equation} \label{axiom.normalization}  f _{l} ^{*} \overline {qc}  =
 \textbf{1} + f_l ^{*} c_1 \cdot q,
\end{equation}
in other words, it is the total Chern class of the canonical line bundle $E \to
\mathbb{CP} ^{n}$, except for the multiplication of $c_1 (E)$ by $q=[pt] \in
\ring$.
%   then for $P=[i_B \circ f _{l}]$ we
%  have: \begin{equation} c ^{q, P} = \textbf{1} - PD
%  ( [ \mathbb{CP} ^{n-1}])  \otimes [pt] \otimes e ^{A}.
% \end{equation} 
\end{itemize}
\end {theorem}

\begin{remark} We note that these axioms are equivalent to the longer
but probably more familiar looking (and slightly simpler in practice) axioms
for stable characteristic classes: For a CW complex $B$, the classes $qc
^{\infty} _{k}$ are a sequence of functions from the reduced $K$-theory group
$\widetilde{K} (B)$ to $H ^{2k} (B, \ring)$ satisfying:
\begin{itemize}
  \item $qc ^{\infty}_{0} (E) = \textbf{1}$ for any $E \in \widetilde{K} (B)$. 
  \item  \label{functoriality/}  For $E_2 \in \widetilde{K} (B_2)$ and a map $g:
  B_1 \to B_2$: $$ g ^{*} qc ^{\infty} _{k} (E_2) = qc  ^{\infty}_{k} g ^{*} (E
  _{2}). $$
\item \label{thm.whitney.sum.first} If $E= E_1
\oplus E_2 \in \widetilde{K} (B)$, then
\begin {equation*} qc  ^{\infty}( E) = qc  ^{\infty}( E _{1}) \cup qc
^{\infty} (E _{2}), \end {equation*} where $ \cup$ is the cup product of cohomology classes with
coefficients in  $\ring$ and $qc ^{\infty} (E)$ is
the total characteristic class
\begin{equation*}
qc ^{\infty} (E) = qc ^{\infty}_0 (E)+\ldots+ qc ^{\infty}_m (E) \ldots \in H
^{*} (B, \ring),
\end{equation*}
\item 
 If $E _{l} \in \widetilde{K} ( \mathbb{CP} ^{n})$ is the class of the canonical
 line bundle  then 
 \begin{equation*} qc ^{\infty} (E _{l}) = \textbf{1} + c _{1} (E _{l}) \cdot q.
\end{equation*}

\end{itemize}
\end{remark}

The first property  in Theorem \ref{thm.axioms} follows immediately from 
the corresponding property 
for the classes $qc _{k} \in H ^{2k} (\Omega SU (n), QH ( \mathbb{CP} ^{n-1}))$. 
Verification of the second property is done in
Section \ref{section.proofs}. 

% As already mentioned  under the Bott isomorphism $BU \simeq
% \Omega SU$, W are identified in the homotopy category.
% Note that by naturality this
% determines all quantum classes of $P _{f _{l}}$. 
% Thus, the above axioms are the axioms for ``Chern classes `` with coefficient
% ring $\ring$  and a very  different normalization. We may still think of them
% Chern classes because they are characterized by logically identical set of axioms.   
Consequently, we see that classes $qc ^{\infty} _{k}$ are stable characteristic
classes and formally have the same axioms as Chern classes. In particular the
splitting principle on $ BU$ allows us to write down quantum classes in terms of Chern classes, we do this below.
\begin{notation} From now on we drop the superscript $ \infty$ in $qc ^{\infty}
_{k}$, as we will deal exclusively with these stable classes, unless specified
otherwise.
\end{notation}
\begin{proof} [Proof of Theorem \ref{theorem.chern.classes}] 
By the
universal coefficient theorem $H ^{*} (BU \simeq \Omega SU, \ring)$ is a free polynomial algebra over 
$ \ring$ on generators $c_k$, (the Chern classes) since this is true over $ \mathbb{Z}$. 
The theorem will follow as soon as we  show that $qc _{k}$ are algebraically independent.
  Let $f: B ^{2k} \to BU(n) \subset BU$ be a cycle and $s: F \to B$ be the splitting map for the 
  associated vector bundle $E _{f}$, so that $[s \circ f] = [f_1] \star [f_2]
  \ldots \star [f_n] $, with $f _{i}$ classifying stabilized line bundles. Set $$x _{i} = f _{i}^
   {*} qc _{1} = f _{i} ^{*} c _{1} \cdot q,$$ then by the axioms we have that $(s \circ f) ^{*} qc _{k}$ is an elementary symmetric
polynomial  $sp _{k} [x_1, \ldots, x_n]$ of degree $k$ on generators $x _{i}$ . 
Clearly $sp _{k} [x_1, \ldots, x_n] = sp _{k} [f_1 ^{*} c _{1}, \ldots, f _{n}
^{*} c _{1}] \cdot q ^{k}$ but the latter is
 by the the axioms for Chern classes  exactly $(s \circ f) ^{*} c
_{k} \cdot q ^{k}$. Since $s ^{*}$ is injective it follows that $qc _{k} = c
_{k} \cdot q ^{k}$, in particular $qc _{k}$ are determined by the axioms,
moreover  since Chern classes are algebraically independent so are the classes $qc _{k}$. (Algebraic independence for both
 series of classes is directly deducible from algebraic independence of
elementary symmetric polynomials, in the ring of all symmetric polynomials, which is their
fundamental property.) 
\end{proof}
\section {Main arguments} \label{section.proofs}
We are going to give two proofs of Proposition \ref{thm.induction}. The first
one is more difficult, but less transcendental of the two, and we present it
here because it has the advantage of suggesting a route to computation of quantum
classes in the non-stable range and because the energy flow picture in this
argument is used for the proof of Theorem \ref{thm.Hofer}. Unfortunately, this
first proof is not very elementary and requires some basics of the theory of
loop groups. The second proof was suggested to me by Dusa McDuff, and basically
just uses the  automatic transversality result of Hofer-Lizan-Sikorav.  For
convenience, these will be presented independently of each other, so can be read in any order. \begin{proof} [Proof of Proposition \ref{thm.induction}] We will need to use the
%  \begin{proposition} \cite[Proposition 8.8.1]{Segal}  There are natural maps $i
%  _{B} ^{n}: \Omega SU (n) \to BU (n)$ inducing isomorphism of homotopy groups up
%  to dimension $2n-2$.
%  \end{proposition}
 cellular decomposition of $ \Omega SU (n)$ induced by the energy
functional $E$, given by a bi-invariant metric on $ SU (n)$.
For various details in the following, the reader is referred to \cite[Sections
8.8, 8.9, 8.10]{Segal}. It is shown in \cite{Segal} that there is a holomorphic
cell decomposition of the Kahler manifold $\Omega SU (n)$ up to dimension $2n-2$,
with cells indexed by homomorphisms $\lambda: S ^{1} \to T$ with all weights
either $1, -1, 0$, (the weights can all be zero) where $T \subset SU(n)$ is a
fixed maximal torus. The closure $ \overline{C}_{\lambda}$ of each such cell $C_
\lambda$, is the closure of unstable manifold (for $E$) of the cycle of
subgroups in the conjugacy class of $ \lambda$ contained in the $E$ level set of
$\lambda$. Let $$f _{\lambda}: B ^{2k}_{\lambda} \to \Omega SU (n)$$ be the
compactified $2k$-pseudocycle representing $ \overline{C} _{\lambda}$. We
may assume now that $k>0$ so that some weights of $\lambda$ are non-zero, since
the case of $k=0$ is excluded in the hypothesis of our proposition.

It is shown  that each $\gamma \in \overline{C}_\lambda$ is
a polynomial loop, i.e. extends to a holomorphic map
$\gamma _{ \mathbb{C}}: \mathbb{C} ^{\times} \to SL _{ \mathbb{C}}{(n)}$.
Consequently, $X _{\gamma}$ has a natural holomorphic structure, so 
that the complex structure on each fiber $M _{z} \hookrightarrow X _{\gamma} \to
\mathbb{CP}^1$ is tamed by $\omega$, and so we have a natural admissible family of complex
structures on $P _{f _{\lambda}}$. By the classical Birkhoff-Grothendieck
theorem $X _{\gamma}$ is biholomorphic to $X _{\lambda_0}$ for some $S ^{1}$ subgroup
$\lambda_0$. Moreover by  \cite[Propostion 8.10.1]{Segal}
$\lambda_0$ is in fact $\gamma_{\infty} : S ^{1} \to SU (n)$, $\gamma_{\infty} \in
\overline{C} _{\lambda}$ is  the $S ^{1}$ subgroup, which is the limit in
forward time of the negative gradient flow trajectory of $\gamma$. In particular
$\gamma_{\infty}$ has all weights either $1, -1, 0$. 
For $\gamma = f _{\lambda} (b)$ we call the corresponding $\gamma_{\infty}$:
\begin{equation*} \label {eq.lambda'} \gamma^b_{\infty}.
\end{equation*}

%  It should 
% moreover be true that the bundle $P _{f _{\lambda}}$ is a stratified holomorphic bundle, by
% \textbf{reference}, \cite[Proposition 8.10.2]{Segal}, but this is also not important to us. 
Aside from the condition on weights of the circle subgroups, here is another
point where the stability condition $2k \leq 2n-2$ comes in. 
 The virtual dimension of $\mathcal {M} (P _{f _{\lambda}}, d,
\{J _{b}\})$ (see \eqref{eq.dim}) is 
\begin{equation*} 2n-2 + 2k + 2d \cdot n  \leq 2n-2 + 2n-2+2d \cdot n < 0
\text { unless $d \geq -1$}.
\end{equation*} 
 Thus only $d=-1$ classes can  contribute to $qc ^{k} [f _{\lambda}]$,
 ($d=0$ can only contribute to $qc _{0} [f _{\lambda}]$; which can be checked by
 analyzing the definition, and $d>0$ results in too high a virtual dimension). 

Let us first understand the spaces of holomorphic $d=-1$ sections in $X
_{\lag}$, with $\gamma_{\infty}$ an $S ^{1}$ subgroup of the type above.
\begin{lemma} \label{lemma.identified} $X _{\lag}$ is
biholomorphic to $S ^{3} \times _{\lag} \mathbb{CP} ^{n-1}$
 i.e. the space of equivalence classes of tuples 
\begin{equation*}  [z_1, z_2, x], \quad
x
\in
\mathbb{CP} ^{n-1},
\end{equation*}
under the action of $S ^{1}$
\begin{equation*} \label {eq.action} \theta\cdot [z_1, z_2, x] = [ e ^{
2 \pi i \theta } z_1, e ^{2 \pi i \theta} z_2, \gamma_{\infty} (\theta) x],
\end{equation*}
using complex coordinates $z_1, z_2$ on $S ^{3}$.
\end{lemma}
\begin{proof} To see this, write $ [z_1,
z_2; x]$ for the equivalence class of the point $$ (z_1/r, z_2/r; x) \in S ^{3}
\times \mathbb{CP} ^{n-1},$$ where $r$ is the norm of $(z_1, z_2)$. We identify
$ D_0 \times \mathbb{CP} ^{n-1}$ with $$ \{ [1,z;x]: |z| \leq 1, x \in
\mathbb{CP} ^{n-1}\}$$ via orientation reversing reflection (that is $(z,x)
\mapsto [1, \bar {z}, x]$) and $D_\infty \times \mathbb{CP} ^{n-1}$ with $$
\{[z, 1; x]: |z| \leq 1, x \in \mathbb{CP} ^{n-1}\}$$ naturally. The gluing map is then \begin{equation*} [1, e
^{-2 \pi i \theta}; x] \sim [e ^{2 \pi i \theta}, 1; \gamma _{\infty}
(\theta)x], \end{equation*} which is consistent with the definition of $X _{\gamma _{\infty}}$, (by using $ \gamma _{\infty}$ as
clutching map). 
\end{proof} 

Let $H_{\gamma_{\infty}}$ be the normalized generating function for the action
of $\gamma_{\infty}$ on $ \mathbb{CP} ^{n-1}$. Each $x$ in the max level set  $F
_{\max}$ of $H_{\gamma_{\infty}}$, gives rise to a section $\sigma _{x}= S ^{3} \times
_{\lag} \{x\}$ of $X _{\gamma_{\infty}}$. It can be checked that $ [\sigma
_{x}]$ has degree $d=-1$. To see this use the expression $N\sigma _{x}= S ^{3}
\times _{\lag} T _{x} \mathbb{CP} ^{n-1}$ to evaluate $c_1$ of the normal
bundle of this section; this Chern number will be $-n$ so long as  action of $
\gamma _{\infty}$ has all weights $1,-1$ or $ 0$ as assumed. (With our somewhat
unusual orientation conventions.)

 Moreover, by elementary energy considerations it can
be shown that these are the only  holomorphic class $d=-1$ sections (or section
class stable maps) in $X _{\gamma _{\infty}}$, see for example the discussion
following Definition 3.3. in \cite{virtmorse}. Consequently the moduli space $\mathcal {M}
= \mathcal {M} (P _{f_\lambda}, d=-1, \{J _{b}\})$ is compact. It's restriction
over open negative gradient trajectories $ \mathbb{R} \to \bar {C} _{\lambda}$ asymptotic to
$\gamma_{\infty}$ in forward time is identified with $ \mathbb{R} \times F _{\max}$,
where $F _{\max}$ is as above. 

 The regularized moduli space can be constructed from $ \mathcal {M}$ together
 with kernel, cokernel data for the  Cauchy Riemann operator for
 sections $u \in \mathcal {M}$. This is somewhat well understood by now, see
 for example an algebro-geometric approach  in \cite{Virt.Mod.cycle}. It is however
 interesting that our moduli space is highly singular, so a direct computation
 of the invariant from this data is not elementary.

Consider the pushforward:
$i\circ f_{\lambda}: B _{\lambda} \to \Omega SU (m)$.
What we need to show is that  $ \mathcal {M}$ is identified with $\mathcal {M}
^{s}= \mathcal {M} (P _{i \circ f_\lambda}, {-1}, \{J _{b}\})$ and the
kernel/cokernel data of the family of CR operators associated with $ \mathcal {M}$ are identified with
the kernel/cokernel data for the family of CR operators associated to $\mathcal
{M} ^{s}= \mathcal {M} (P _{i \circ f_\lambda}, {-1}, \{J _{b}\})$.

The fact that the moduli spaces $ \mathcal {M}$, $ \mathcal {M} ^{s}$ are identified follows more or less immediately from the preceding discussion. The inclusion map 
$i:\Omega SU (n) \to \Omega SU (m)$ takes gradient trajectories in $\Omega SU
(n)$ to gradient trajectories in $\Omega SU (m)$  and any circle subgroup
is taken to a circle subgroup  with $m-n$ new 0-weights. 
There is a natural map $j: P _{f} \to P _{i \circ f _{\lambda}}$, and its clear
that it identifies $\mathcal {M}$ with $ \mathcal {M} ^{s}$.
We will denote $j(X_b)$ by $ X _{j(b)} ^s$,  and an element in $ \mathcal {M}
^{s}$ identified to an element $u = (\sigma _{x}, b) \in \mathcal {M}$ by $j(u)
$, where $b \in B _{f_\lambda}$.

We show that the CR operators at $u$, $j(u)$ have the same kernel and
cokernel. Let $V$ denote the infinite dimensional domain for the 
 CR operator $D _{u}$ at $u$. There is a subtlety here, since our
target space $P _{f _{\lambda}}$ is a smooth stratified space we actually have to
work a stratum at a time, but we suppress this. The domain for the 
CR operator $D _{j(u)}$ is the appropriate Sobolev completion of the space
of $C ^{\infty}$ sections of the  bundle $j(u) ^{*} T ^{vert}X _{j (b)} ^{s}$,
where
\begin{equation*} j (u) ^{*} T ^{vert} X _{j (b)}^{s} \simeq S^{3} \times_
{\gamma_{\infty} ^{ j(b)}} T _{j(x)} \mathbb{CP} ^{m-1},
\end{equation*}
is a holomorphic vector bundle, and $\simeq$ is isomorphism of holomorphic
vector bundles. 
Similarly, 
\begin{equation*} u ^{*} T ^{vert}X _{b} \simeq S^{3} \times _{\gamma
^{b}_{\infty}} T _{x}\mathbb{CP} ^{n-1}.
\end{equation*}

% , $NX$ will denote the subbundle of $TX ^{s}$,
% respectively $TX$ of vectors tangent to fibers of projection $X ^{s} \to \mathbb{CP}^1$, respectively $X \to \mathbb{CP}^1$. 

Consequently  $j (u) ^{*} T ^{vert}X _{j (b)} ^{s}$ holomorphically splits into
line bundles with Chern numbers either $0, -2, -1$.  
 All the Chern number
 0 and $-2$ summands are identified with corresponding summands of $u ^{*}
 T ^{vert}X_{b}$, and consequently the domain of $D _{j(u) }$ in comparison to
 $D _{u}$ is enlarged by the space of sections $W$, of sum of Chern number $-1$
holomorphic line bundles. But in our setting since $J _{b}$ are all integrable,
$D _{j(u)}$ can be identified with Dolbeault operator and so we will have no
``new kernel or cokernel'' in comparison to $D _{u}$. More precisely:  $$D _{j(u)}: W \to \Omega ^{0,1} (S
^{2}, j(u) ^{*} T ^{vert}X _{j (b)} ^{s})/ {\Omega ^{0,1} (\mathbb{CP}^1, u ^* {T
^{vert}X})}$$ is an isomorphism, which concludes our argument.
\end {proof}
\begin{proof} [Second proof of Proposition \ref{thm.induction}] It is enough to
verify the stabilization property for $m=n+1$. Let $f: B ^{2k} \to
\Omega SU (n)$ be a map of a smooth, closed oriented manifold.
% \widetilde{A} = [\sigma _{const}] - d \cdot [line]\in H_2 (X)$$
By our assumptions the virtual dimension of  $ \mathcal {M} (P _{f}, d,
\{J _{b}\}) $ is
\begin{equation*}  2n-2+ 2k + 2d \cdot n \leq 2n-2 + 2n-2 + 2d \cdot n <
0 \text { unless $d \geq -1$},
\end{equation*} 
Thus, only $d=-1$ degree holomorphic sections can contribute to
$qc_k [{f}]$, as $d=0$ only contributes to $qc_0$ and $d >0$ results in too high 
virtual dimension.

Let $j: \Omega SU (n) \to \Omega SU (n+1)$ be the inclusion map, and denote 
$j \circ f$ by $f'$.  
% Clearly, the structure group of $ \mathbb{CP} ^{n-1}
% \hookrightarrow P _{f'} \to B \times \mathbb{CP}^1$, can be reduced to the group $PSU
% (n) \subset PSU (n+1)$. Consequently we may take a $PSU (n)$
% connection $ \mathcal {A}$ on $P _{f}$, and extend it to a connection 
% on $P _{f'}$ preserving the image of the natural embedding $P _{f} \to P _{f'}.$
Take a family of almost complex structures $ \{J _{f,b} \}$ on $P _{f}$ for
which the moduli space $ \mathcal
 {M} (P _{f}, -1, \{J _{f,b} \})$ is regular. Extend $ \{J ^{} _{f, b}\}$ to a
 family $ \{J ^{} _{f',b}\}$ on $P _{f'}$ in any way. 
Consequently, for the families of almost complex structure $ \{J
_{f,b}\}$ on $P _{f}$ and $ \{J _{f', b}\}$ on $P _{f'}$ the natural embedding
of $P _{f}$ into $P _{f'}$ is holomorphic. The intersection number of a curve
$u \in \mathcal {M} (P _{f'}, -1, \{J _{f',b}\})$ with $P _{f} \subset P
_{f'}$, is the intersection number of $-[line]$ with $ \mathbb{CP} ^{n-1} \subset
\mathbb{CP} ^{n}$, i.e. it is -1. Consequently, by a version of positivity of
intersections, which follows by a Theorem of Micallef-White 
\cite{Micallef.white} (this is also exercise 2.6.1 in \cite{MS}) for the family $
\{J _{b}\}$ all the elements of the space $ \mathcal {M} (P _{f'}, -1, \{J _{b}\})$
  are contained inside the image of embedding of $P _{f}$ into $P _{f'}$. We now show that $ \{J _{f',b}\}$ is also regular. This will immediately
 yield our proposition. The pullback of the normal bundle to
 embedding of $P _{f}$, by $u _{b} \in \mathcal {M} (P _{f'}, -1, \{J
 _{f',b}\})$ is the degree $ -1$ line bundle $ \mathcal {O} (-1)$. So we have an
 exact sequence \begin{equation*}  u _{b} ^{*} T ^{vert} P _{f} \to (j \circ u
 _{b}) ^{*} T ^{vert} P _{f'}  \to \mathcal {O} (-1).
\end{equation*}
 By construction of $ \{J
 ^{pert} _{f', b}\}$ the real linear CR operator $D _{i \circ u}$ is compatible
 with this exact sequence. More explicitly, we have the real linear CR
 operator
\begin{equation*} \Omega ^{0} ( \mathbb{CP}^1, \mathcal {O} ({-1})) \to \Omega
^{0,1} ( \mathbb{CP} ^{1}, \mathcal {O} (-1))
\end{equation*}
induced by $$D _{j \circ u _{b}}: \Omega ^{0} (\mathbb{CP}^1, (j \circ u) ^{*} T
 ^{vert} P _{f'}/ u ^{*} T ^{vert} P _{f}) \to \Omega ^{0,1} (\mathbb{CP}^1,  (j \circ
 u) ^{*} T ^{vert} P _{f'}/ u ^{*} T ^{vert} P _{f}),$$
  since $u ^{*} T ^{vert} P _{f}
 \subset  (j \circ u) ^{*} T
 ^{vert} P _{f'}$ is $J _{f',b}$ invariant.  Such an operator is surjective by
 the automatic transversality theorem of Hofer-Lizan-Sikorav
 \cite[Theorem 1]{HLS}.
\end{proof}
\begin{proof} [(Verification of part 2 of Theorem \ref{thm.axioms})] 
Consider the space $Z$ of $ S ^{1}$-subgroups of $SU (k+2)$, conjugate to the
diagonal $S ^{1}$-subgroup $ \lambda$ with diagonal entries $ [e ^{2\pi i
\theta}, e ^{- 2\pi i\theta}, 1,\ldots, 1]$, and such that the weight 1 space
(the subspace on which $ S ^{1}$ acts by $ e ^{2\pi i \theta}$)  is fixed for
all elements of $Z$. Then $z \in Z$ is completely determined by the choice of -1 weight space,
i.e. by choice of a  line in $ \mathbb{C} ^{k+1}$. So $Z$ is identified with $
\mathbb{CP} ^{k}$. So we have an inclusion $$\lambda: \mathbb{CP} ^{k} \to
\Omega SU (k+2).$$ 
By  \cite[Proposition 8.8.1] {Segal}, there is a natural map (the
Bott map) $$i _{B}
^{k+2}: \Omega SU (k+2) \to BU,$$ which induces an isomorphism on homotopy
groups up to dimension $2k+2$, and consequently on rational homology groups as
well by Theorem \ref{thm.milnor.moore}. 
\begin{lemma} $$i _{B}
^{k+2}: \Omega SU (k+2) \to BU (k+2),$$ restricted to $Z$ is exactly the
natural map $f_l: \mathbb{CP} ^{k} \to BU (k+2)$, classifying the reduced $K$-theory class of the canonical line bundle $E _{l}$. 
\end{lemma}
\begin{proof} 
Let us temporarily shift degree here down by two, so we will be proving the
lemma for $i _{B} ^{k}$. We'll be essentially following \cite{cohen}, except that
we work with negative winding number loops, instead of positive winding number.
There is a filtration of $\Omega SU (k)$ by complex compact subvarieties \begin{equation} \label {eq.filtration} F
_{-1,k} \hookrightarrow \ldots \hookrightarrow F _{w,k} \hookrightarrow F
_{w-1,k} \hookrightarrow \ldots \hookrightarrow F _{-\infty, k} \hookrightarrow 
\Omega _{pol} SU (k) \simeq \Omega SU (k),
\end{equation}
where $\Omega _{pol} SU (k)$ is the subgroup of $\Omega SU (k)$ consisting of
loops whose entries in the matrix representation are finite Laurent polynomials
in $z$, ($z (\theta)= e ^{2\pi i \theta}$) that is: \begin{equation} \gamma
(\theta) = \sum _{j = -N} ^{N} A _{j} z ^{j} (\theta),
\end{equation}
for $A _{i}$ $k \times k$ matrix, (this is shown in \cite{Segal} to be homotopy
equivalent to the smooth loop space $\Omega SU (k)$).  We have a subvariety $ \widetilde{F}
_{w,k} \subset \Omega U (k)$ defined as the subgroup consisting of polynomial
loops with $A _{i}= 0$ for $i  >0$, and winding number $w$ (degree of
composition with the determinant map). We then define $F _{w,k} = \lambda_1 ^{-w} \widetilde{F} _{w,k} \subset \Omega SU (k)$, where $ \lambda_1$ is the winding number one $S
^{1}$-subgroup with diagonal entries $ [e ^{2 \pi i \theta}, 1, \ldots, 1]$. Now
Proposition 1.3 in \cite{cohen} following Mitchell \cite{mitchell} states that 
 $\lim _{k \mapsto \infty} F
_{w, k} \simeq BU (-w)$ and that the following diagram homotopy commutes:
\begin{equation*}   
\xymatrix{F _{w,k} \ar [d]  \ar[r] ^{k \mapsto \infty} & BU (-w) \ar[d]\\  
\Omega SU (k) \ar [r] ^{i _{B } ^{k}} & BU}.
\end{equation*} 
(There appear to be some minor typographical errors in \cite{cohen} to the 
effect that $\Omega U (n)$ should be replaced by $\Omega SU (n)$ in a few
instances.) Also $ \widetilde{F} _{-1,k}$ is readily seen to be the space of $S
 ^{1}$-subgroups conjugate to weight $-1, 0, \ldots, 0$ subgroup, this
space is just $ \mathbb{CP} ^{k-1}$, since we are just choosing a weight -1 line
in $ \mathbb{C} ^{k}$. Consider the inclusion $i: \widetilde{F} _{-1,k-1} \to
\widetilde{F} _{-1,k}$ induced by the inclusion  $SU (k-1) \hookrightarrow SU (k)$, 
\begin{equation*} 
A \mapsto \left(\begin{array}{ccc} 1 & \\
& A
\end{array} \right).
\end{equation*}
This gives us the inclusion:
 \begin{equation*} \mathbb{CP} ^{k-2} \simeq \widetilde{F}
_{-1, k-1} \xrightarrow{i} \widetilde{F} _{-1,k} \to \Omega SU (k).
\end{equation*}
The image of this inclusion is our subspace $Z ^{k-2} \subset
\Omega SU (k)$. On the other hand by
the above stated Mitchell's theorem, and since we are in the stable range ($deg
Z \leq 2k-2$, see explanation of $i _{B} ^{k}$), passing to the limit in $k$
this inclusion is identified in the homotopy category with the natural map $ \mathbb{CP} ^{k-2} \hookrightarrow BU (1) \hookrightarrow BU$, classifying the reduced K-theory class of the tautological
line bundle.  
\end{proof}

The fiber of the associated bundle $P _{f_l} \to \mathbb{CP} ^{k}$ over $z \in
\mathbb{CP} ^{k}$ is $X _{\lambda (z)}$, 
which by Lemma \ref{lemma.identified} can be identified with $S^3 \times _{S^1}
\mathbb{CP} ^{k+1} $, where $S^1$ acts diagonally on $S ^{3} \times \mathbb{CP} ^{k+1}$ by $$ \theta \cdot
(z_1, z_2; x) = (e ^{2 \pi i \theta} z_1, e ^{2 \pi i \theta} z_2; \lambda_z
(\theta)x),$$ using complex coordinates on $S^3$. 
% To see this, write $ [z_1,
% z_2; x]$ for the equivalence class of the point $ (z_1/r, z_2/r; x) \in S ^{3}
% \times \mathbb{CP} ^{k+1}$, where $r$ is the norm of $(z_1, z_2)$. We identify
% $ D_0 \times \mathbb{CP} ^{k+1}$ with $ \{ [1,z;x]: |z| \leq 1, x \in
% \mathbb{CP} ^{k+1}\}$ naturally and $D_\infty \times \mathbb{CP} ^{k+1}$ with $
% \{[z, 1; x]: |z| \leq 1, x \in \mathbb{CP} ^{k+1}\}$ via the orientation
% reversing reflection. The gluing map is then \begin{equation*} [1, e ^{2 \pi i \theta}; x] \sim [e ^{-2 \pi i \theta}, 1;
% \lambda _{z} (\theta)x], \end{equation*} which is consistent with
% the definition of $X _{\lambda (z)}$.
So each fiber has a natural holomorphic
structure $J_z$, which varies smoothly in the family, and in fact all fibers are
biholomorphic. 

Let $ \sigma ^{z}_{\max}$ denote the holomorphic section of $X _{\lambda (z)}$
corresponding to the maximal fixed point $ \max$ of the Hamiltonian $S
^{1}$-action of $ \lambda (z)$ on $ \mathbb{CP} ^{k+1}$ as follows:
\begin{equation} \sigma ^{z} _{\max} = S ^{3} \times _{\lambda (z)} \{\max\}
\subset X _{\lambda (z)}.
\end{equation}
For example: for the diagonal subgroup $\lambda (\theta) = [e ^{2 \pi i \theta},
e ^{-2 \pi i \theta}, 1, \ldots, 1]$ acting on $ \mathbb{CP} ^{k+1}$, 
$\max= [1,0,\ldots,0]$. As we previously mentioned these are the only degree
$d=-1$ holomorphic sections (or section class stable maps) of $X _{\lambda
(z)}$. In particular the moduli space  $ \mathcal {M}=\mathcal {M} (P _{f
_{l}}, d=-1, \{J _{z}\})$ is compact and is identified with the base $
\mathbb{CP} ^{k}$.

The normal bundle  $N _{z} = (\sigma ^{z} _{\max}) ^{*} T ^{vert} X _{\lambda
(z)}$ is identified as
\begin{equation*} S ^{3} \times _{\lambda (z)} T _{\max} \mathbb{CP} ^{k+1}.
\end{equation*}
The action of $ \lambda (z)$ on $ T _{\max} \mathbb{CP} ^{k+1}$ has
weights $-2,-1,\ldots, -1$. Denote the weight -2 subspace by $ \mathcal {O} _{z}
(-2)$. The cokernel of the corresponding  CR operator,  $$D_{x,z}: \Omega ^{0}
(\mathbb{CP}^1, N _{z}) \to \Omega ^{0,1} (\mathbb{CP}^1, N_z)$$ is identified with the cokernel of the Dolbeault
operator $$H^{0,1}_{\bar \partial} ( \mathbb{CP}^1, \mathcal {O}_z (-2)) \simeq (H ^{1,0}_
{ \bar \partial} (\mathbb{CP}^1, \mathcal {O} _{z} (-2) ^{*} )) ^{*}.$$ By a special case of Kodaira-Serre duality, the latter can be identified with $$L_z=(H ^{0} (\mathbb{CP}^1,
\mathcal {O} _{z} (-2)^{*} \otimes K _{z})) ^{*},$$ where $K _{z}=T^* (\sigma
^{z} _{\max})$ denotes the canonical bundle of $\sigma ^{z} _{\max}$. 
The vector space $L _{z}$ is canonically identified with $ \mathcal {E} _{z}
\otimes \mathcal {K} _{z}$, where $ \mathcal {E} _{z}$ is the restriction
$\mathcal {O} _{z} (-2) |_0$ and likewise $ \mathcal {K} _{z} =  K
_{z} ^{*}|_0$. This is because $ (\mathcal {O} _{z} (-2)^{*} \otimes K
_{z})) ^{*}$ has Chern number 0 and so is holomorphically trivial.  

This gives us an obstruction line bundle $O= \mathcal {E} \otimes \mathcal {K}$
over $ \mathcal {M} \simeq \mathbb{CP} ^{k}$. But $ \mathcal {K}$ is trivial,
since each $ \mathcal {K}_z$ is identified with $T  _0 \mathbb{CP}^1$ under natural projection $\pi: X _{\lambda
(z)} \to \mathbb{CP}^1$. On the other hand,  $  \mathcal {E} _{z}$ is
naturally identified with weight $ -1$ subspace of the action of $ \lambda (z)$
on $ \mathbb{C} ^{k+2}$, i.e. with the fiber of $E _{l}$ over $z$, where $E
_{l}$ was the line bundle whose reduced $K$-theory class was classified by
$f_l: \mathbb{CP} ^{k} \to BU$. Consequently the regularized moduli space is
given by intersection of a generic section of $O$ or $E _{l}$ with the $
0$-section, i.e. its homology class is the homology Euler class of $ E _{l}$. 
It follows that $f _{l} ^{*}qc_i  [ \mathbb{CP} ^{i}]  = 0$ unless
$i=1$ and $ f _{l} ^{*}qc_1 [ \mathbb{CP} ^{1}]  =  c_1 (E
_{l}) [ \mathbb{CP} ^{1} ]  \cdot q=  -q \in QH (
\mathbb{CP} ^{k+1})$.
\end{proof}

\begin{proof} [Proof of Theorem \ref{thm.Hofer}]
% The condition $0 \leq i-1 \leq 2n-2$, means that we need to work with image of
% $H _{k} (\Omega SU (n))$ in $H_k( \mathcal {L} SU (n))$,
% with $$0 \leq k \leq 2n-2.$$ 
% We need the following fact, which we already used previously: the map $H _{2k}
% (\Omega SU (n)) \to H _{2k} (\Omega SU)$, is an isomorphism in the range $0
% \leq 2k \leq 2n-2$, (see \cite[Proposition 8.8.1]{Segal}).  
% and that the classes
% $qc_k \in H _{2k} (\Omega SU, \ring)$ naturally extend to $H _{k} ( \mathcal {L}
% SU, \ring)$, (in fact to homology $S ^{1}$-Borel quotient). This latter fact readily follows from
% the construction in \cite{GS}, and is already outlined here in slightly
% different formulation: we discussed $\Omega ^{2} BSU _{S ^{1}}$.
Let $0 \neq a \in H _{2k} (\Omega SU (n))$, with $0
 \leq 2k \leq 2n-2$.  By Corollary
\ref{theorem.corollary} we have that
%  for a
% smooth cycle $$f: B ^{2k} \to \Omega SU (n),$$ representing class $a$, 
% some quantum Chern number is non-vanishing:
 \begin{equation} \label {eq.nonvanish} \langle \prod _{i} qc ^{\alpha_i}
 _{\beta_i}, a\rangle \neq 0, 
\end {equation}
for some $\alpha _{i}$, $\beta _{i}$.
Since these classes are pull-backs of the classes $$qc _{k} \in H _{2k} (\Omega
\text {Ham}( \mathbb{CP} ^{n-1}, QH ( \mathbb{CP} ^{n-1})),$$ the same
holds for these latter cohomology classes and the cycle $i_*a$. 
\begin{lemma} For any representative $f: B \to \Omega \text {Ham}( \mathbb{CP}
^{n-1}, \omega)$ for $i_*a$,
 $$L ^{+} (f (b)) \geq 1$$ for some $b \in B$.
\end{lemma}
\begin{proof} 
This is essentially \cite[Lemma
3.2]{virtmorse}, and we reproduce it's proof here for convenience. 
The total space of $P _{f}$ is 
\begin{equation}  P _{f}= B\times \mathbb{CP} ^{n-1}
\times D ^{2} _{0} \bigcup B \times \mathbb{CP} ^{n-1} \times   D ^{2}
_{\infty}/ \sim, 
\end{equation}
where $ (b, x, 1, \theta)_0 \sim (b, f _{b, \theta}(x), 1, \theta)
_{\infty}$,
using the polar coordinates $(r, 2 \pi
 \theta)$.
The fiberwise family of Hamiltonian connections $ \{ \mathcal {A} _{b} \}$  are  
induced by a family of certain closed forms $ \{ \widetilde{\Omega}  _{b}\}$, which
we now describe, by declaring horizontal subspaces of $ \mathcal {A} _{b}$ to be
$ \widetilde{\Omega} _{b}$-orthogonal to the vertical subspaces of $\pi: F _{b}
\to \mathbb{CP}^1$, where $F _{b}$ is the fiber of $P _{f}$ over $b \in S ^{k}$.
 
The construction of this family mirrors the construction in
 Section 2.2 of \cite{GS}.  First we define a
family of forms $\{{ \widetilde{\Omega}} _{
b} \}$ on $B\times \mathbb{CP} ^{n-1} \times   D ^{2} _{\infty}$.
%  For
% convenience and to be consistent with \cite{GS} we identify
% $(\widehat{\mathbb{C}}/D ^{2} ) _{\infty}$ with $D ^{2}_ \infty$  via an
% orientation reversing reflection (so that it has the opposite orientation) and
% define:
\begin{equation} \label {eq.coupling.family} { \widetilde{\Omega}}_{b} |
_{D ^{2} _{\infty}} (x,r, \theta) = \omega - d ( \eta (r) H^b_ \theta (x))
\wedge d\theta \end{equation} Here, $H ^{b}_\theta$ is the generating
Hamiltonian for $f(b)$, normalized so that $$\int _{ \mathbb{CP} ^{n-1}}
H ^{b}_\theta \omega ^{n-1}=0, $$ for all $\theta$ and the function $\eta:
[0,1] \to [0,1]$ is a smooth function satisfying $$0 \leq \eta' (r),$$
and $$ \eta (r) = \begin{cases} 1 & \text{if } 1 -\delta \leq r \leq 1 ,\\
r ^{2}  & \text{if } r \leq 1-2\delta,
\end{cases}$$  for a small $\delta >0$.

It is not hard  to check that the gluing relation $\sim$ pulls
back the form $ \widetilde{\Omega}_{b}| _{D ^{2} _{\infty}}$ to the form 
$\omega$ on the boundary $ \mathbb{CP} ^{n-1} \times \partial D^2_0$,
which we may then extend to $\omega$ on
the whole of $ \mathbb{CP} ^{n-1}
\times
D^2 _{0}$. Let $\{ \widetilde{\Omega} _{b}\}$ denote the resulting family
on $X _{b}$. The forms $ \widetilde{\Omega}  _{b}$ on $X_{b}$ restrict
to $\omega$ on the fibers $ \mathbb{CP} ^{n-1} \hookrightarrow X _{b} \to
\mathbb{CP}^1$ and the 2-form obtained by fiber-integration $\int _{ \mathbb{CP} ^{n-1}}
(\widetilde{\Omega} _{b}) ^{n}$ vanishes on $\mathbb{CP}^1$. Such forms are called \emph{coupling forms}, which is a
notion due to Guillemin, Lerman and Sternberg \cite{GuileminSternberg}. We then
have a   symplectic form
\begin{equation*} \Omega _{b} = \widetilde{\Omega} _{b} + \max _{x \in
\mathbb{CP} ^{n-1}} H ^{b} _{\theta} \, d \eta \wedge d\theta + \epsilon \cdot
\omega _{st},
\end{equation*}
defined on $ \mathbb{CP} ^{n-1} \times D ^{2} _{\infty}$ trivially extending to
$X _{b}$. Let $J ^{ \mathcal {A}} _{b}$ denote the almost complex structure on
$X _{b}$ induced by $ \mathcal {A}$ by declaring horizontal subspaces to be
invariant, with holomorphic projection map to $ \mathbb{CP} ^{1}$, and
restricting to standard $j$ on the fibers $ \mathbb{CP} ^{n-1} \hookrightarrow
X _{b} \to \mathbb{CP} ^{1}$.

Since $J ^{ \mathcal {A}} _{b}$ is by construction compatible with
$\Omega _{b}$, $J ^{ \mathcal {A}} _{b}$-holomorphic section $u$ of $F _{b}$ in
degree $d <0$ gives rise to a lower bound \begin{equation*} -d= - \langle
[\widetilde{\Omega} _{b}], [u] \rangle \leq \int _{\mathbb{CP}^1} \max _{x \in \mathbb{CP} ^{n-1}} H ^{b} _{\theta} \, d \eta \wedge d \theta + \epsilon= L
^{+}( f (b)) +\epsilon.
\end{equation*}
Of course such a holomorphic section exists  by  
\eqref{eq.nonvanish}. 
Now make $ \epsilon$ tend to $ 0$. 
\end{proof}

% Consequently,; this works exactly as Lemma \ref{lemma.norm} once we note that the family of connections $ \{ \mathcal {A} _{b}\}$ on $P _{f}$ maybe adapted to Hofer
% geometry of the map $f$ in an appropriate way. 
%  constructed with the property that
% $$\int _{\mathbb{CP}^1} |* R^{\mathcal {A} _{b}}| ^+_{H} \,\omega _{st} = L ^{+} (i
% \circ f (b)).$$ Here $|\cdot|_H^+$ is the positive Hofer norm and $* \mathcal R^{
% \mathcal {A}_b}$ is the usual Hodge star of the Lie algebra valued $2$-form $R ^{
% \mathcal {A} _{b}}$.  
 The theorem follows once we note that there is a representative $f': B \to
 \Omega \text {Ham}( \mathbb{CP} ^{n-1}, \omega)$ for $i_*a$, in the form
 $i \circ f$, with $f: B \to \Omega SU (n)$, s.t. the image of $f'=i \circ f$
 is contained in the sublevel set $\Omega ^{1} \text {Ham}( \mathbb{CP} ^{n-1}
 \omega)$ for $L ^{+}$, (loops with $L ^{+}$ length at most $1$). Such a
 representative is  found from the energy flow cellular structure of
 $\Omega SU (n)$, because this cellular structure is perfect (all cells
 are of even dimension) $[f]$ must be some combination $ \sum _{i} a_i [f
 _{\lambda_i}]$, and all $f _{\lambda _{i}}$ lie in $ \Omega ^{1} SU (n)$,
 see the proof of Proposition \ref{thm.induction}.
\end {proof}
\bibliographystyle{siam}  
\bibliography{link} 
\end {document}